\documentclass[11pt]{article}

\usepackage{amsmath, latexsym, amssymb}
\usepackage{graphicx}

\makeatletter
\newfont{\footsc}{cmcsc10 at 8truept}
\newfont{\footbf}{cmbx10 at 8truept}
\newfont{\footrm}{cmr10 at 10truept}
\makeatother
\newcounter{example}

\newtheorem{Remark}{Remark}
\newtheorem{Conjecture}{Conjecture}

\newtheorem{Proposition}{Proposition}

\newenvironment{proof}{\medskip\noindent{\it Proof.\ }}{\hfill \mbox{$\Box$}\medskip}

\title{Some conjectures on the zeros 
of approximates to the Riemann $\Xi$-function and incomplete gamma functions}

\author{J. Haglund
\thanks{Work supported in part by NSF grants DMS-0553619 and DMS-0901467}\\
\small Department of Mathematics\\[-0.8ex]
\small University of Pennsylvania, Philadelphia, PA 19104-6395 \\[-0.8ex]
\small \texttt{jhaglund@math.upenn.edu }}

\date{October $16$, $2009$}
\begin{document}

\maketitle

\bibliographystyle{alpha}

\begin{abstract}
Riemann conjectured that all the zeros of the Riemann $\Xi$-function are real, which is now known as the
Riemann Hypothesis (RH).  In this article we introduce the study of the zeros of the truncated sums 
$\Xi _N(z)$ in
Riemann's uniformly convergent infinite series expansion of $\Xi (z)$ involving incomplete
gamma functions.  We conjecture that when the zeros of 
$\Xi _N (z)$ in the first quadrant of the complex plane are listed by increasing real part, 
their imaginary parts are monotone nondecreasing.  We show how this conjecture implies
the RH, and discuss some computational evidence for this and other related conjectures.
\end{abstract}

\section{Introduction}

Following Riemann (as described in a copy of an English translation of his
memoir contained in the appendix of \cite{Edwards}), let 
\begin{align}
\Xi (z) 
%\footnotemark
= \frac {1}{2}(.5+iz)(-.5+iz)\pi ^{-(.5+iz)/2}\Gamma (\frac{1}{2}
(.5+iz)) \zeta (.5+iz).
\end{align}
$\Xi (z)$ is an even, entire function, 
and the famous Riemann Hypothesis (RH) says that all the zeros of $\Xi$ are real.
Let $Q$ denote the first quadrant of the complex plane $\Re (z) \ge 0, \Im (z) \ge 0$.
Since $\Xi (z)$ is even and real on the real line, we can restate the RH
as saying all zeros of $\Xi (z)$ in $Q$ are real.  
Since $\zeta (s)$ is nonzero in $\Re (s)>1$, it follows that
all zeros of $\Xi$ in $Q$ satisfy $\Im (z) \le .5$.
In $1914$ Hardy (as reprinted in \cite{BCRW}) showed that
$\Xi (z)$ has infinitely many real zeros and in $1942$ Selberg \cite{Sel42} showed that a
positive proportion of the zeros of $\Xi (z)$ are real.  More recent work of
Conrey \cite{Con89} has at least $2/5$ of the zeros on the real line.

Riemann derived the following expression for $\Xi (z)$;
\begin{align}
\label{fourier}
\Xi (z) = \int _{0}^{\infty} \cos (zt) \phi (t) \, dt,
\end{align}
where $\phi (t) = \sum_ {n=1}^{\infty} \phi _n (t)$ with 
\begin{align}
\phi _n (t) = \exp (-n^2\pi \exp (2t))(8\pi ^2n^4 \exp (4.5t) - 12\pi
 n^2 \exp (2.5t)).
\end{align}
The function $\phi (t)$ is known to be an even function of $t$.
P\'olya \cite{Pol26} investigated ways of approximating $\phi (t)$ by simpler functions.
He showed that if in (\ref{fourier}) we replace $\phi (t)$ by 
\begin{align}
{\tilde \phi _1 (t)} = \exp (-\pi \cosh (2t))(8\pi ^2 \cosh (4.5t) - 12\pi
 \cosh(2.5t))
\end{align}
(obtained by replacing most of the
exponentials in the definition of $\phi _1 (t)$ by hyperbolic cosines),
then the resulting integral has only real zeros.  P\'olya also showed that if we replace $\phi (t)$ by any function which is not an even function of $t$, then the resulting integral has only finitely many real zeros.  Hejhal \cite{Hej92} 
investigated what happens if we replace $\phi (t)$ by a ``P\'olya
approximate", i.e. a finite sum of the form
\begin{align}
\sum_{n=1}^N 
\exp (-n^2\pi \cosh (2t))(8\pi ^2 n^4\cosh (4.5t) - 12\pi
 n^2\cosh (2.5t)).
\end{align}
By building on earlier work of Bombieri and Hejhal \cite{BoHe87}, which showed that, under
the Generalized Riemann Hypothesis (GRH) and some other assumptions, 
certain linear combinations of Dirichlet $L$-series asymptotically have $100\%$ of their zeros
on the critical line, Hejhal was able to show that if we replace $\phi (t)$ 
in (\ref{fourier}) by a P\'olya approximate the resulting function asymptotically
also
has $100\%$ of its zeros on the real line (but infinitely many
zeros off the line).  By $100 \%$ asymptotically we mean that
the proportion of zeros in $Q$ satisfying $\Re (z)\le m$ that are on the real line
approaches $1$ as $m \to \infty$.

The starting point for this investigation is the idea that
perhaps it is not necessary for worthwhile approximates to have all their
zeros on the real line.  If a given family of approximates approach $\Xi (z)$ 
uniformly, and if for each element in the family one could prove that within a certain sub-region
of $Q$ all the zeros are real, with the size of the sub-region expanding to
eventually include all of $Q$ as our
approximates approach $\Xi$, then this would also imply RH.  Thus it may be worth studying
replacements for $\phi (t)$ in (\ref{fourier}) which are not even.  With this in mind, a 
natural question to ask is what happens if we replace $\phi (t)$ by 
$\sum_{n=1}^N \phi _n (t)$.

\section{Preliminary calculations}

Let $G(z;a,b)$ denote the integral
\begin{align}
G(z;a,b) = 4\int _{0}^{\infty} \cos (2zu) \exp (2bu -a\exp(2u)) \, \, du,
\end{align}
where $z \in \mathbb C$, $a,b \in \mathbb R$ with $a>0$.  
Making the change of variable $t=a\exp (2u)$, so $dt=a\exp (2u)2 \, du$, 
and $du = \, dt/2t$, we get
\begin{align}
\label{hgamma}
G(z;a,b) &= 4\int _{a}^{\infty}
\exp (b \log (t/a) -t )
\cos (z \log (t/a) ) \, \frac {dt}{2t} \\
&= \int _{a}^{\infty} (t/a)^{b} \exp (-t) \left ( 
\exp (iz\log (t/a)) + \exp (-iz\log(t/a)) 
\right ) \\
&= \int _{a}^{\infty} \exp (-t) \left ( (t/a)^{b+iz} + (t/a)^{b-iz}  \right ) 
\, \frac {dt}{t} \\
&=
\frac {\Gamma (b+iz,a)}{a^{b+iz}} + 
\frac {\Gamma (b-iz,a)}{a^{b-iz}},
\end{align}
where
\begin{align}
\Gamma (z,a) = \int_a^{\infty} \exp (-t) t^{z} \, \frac {dt}{t}
\end{align}
is the (upper) incomplete gamma function.   For lack of a better name, 
we will refer to $G(z;a,b)$ as a ``hyperbolic gamma function".

From (\ref{fourier}) and (\ref{hgamma}) we have
\begin{align}
\label{XiG}
\Xi (z) = \sum_{n=1}^{\infty}2\pi ^2 n^4 G(z/2;n^2\pi,9/4) -3\pi n^2G(z/2;n^2\pi,5/4),
\end{align}
the interchange in integration and summation being justified by the uniform convergence.
For $a \in \mathbb R$, $a>0$, the function $\Gamma (z,a)$ is entire (as a function of $z$), 
and hence so is $G(z;a,b)$.  There is a routine in Maple to compute
$\Gamma (z,a)$, and (in the RootFinding package) a routine to compute the zeros, using the argument principle and Newton's method, of a given analytic
function in any rectangle of the complex plane.  Using this, the author made several computations to compute the zeros of the $\Xi$-approximates
\begin{align}
\label{GXi}
\Xi _N (z) :=
 \sum_{n=1}^N \Phi _{n}(z), 
\end{align}
where
\begin{align}
\Phi _n (z) :=2\pi ^2 n^4 G(z/2;n^2\pi,9/4) -3\pi n^2G(z/2;n^2\pi,5/4),
\end{align}
for various small values of $N$.  
%The author has also made several other runs computing zeros of
%various generic linear combinations of hyperbolic gamma functions. 
Lists of zeros for some of these are
contained in the Appendix.  
%Zeros for some of these
%The zeros for $N=1$, $N=2$, and $N=3$ in the rectangle 
%with opposite corners $(0,-.1)$ and $(100,500)$ are shown in the appendix.   
In these computer runs, the parameter ``Digits" in Maple (which tells the computer to use this
many significant digits in all calculations) was typically set to $20N-10$ or so,
whatever number of digits was needed
to compute the function in question over the specified rectangle accurately to $20$ or so significant digits.
After runs were made first with Digits equal to $20N-10$, they were sometimes
run again with Digits equal to $20N$, and the resulting
zeros typically agreed to at least $16$ decimal digits or so, 
which the author has taken to mean 
the computer generated zeros (for Digits equal to $20N-10$) 
agree with the actual ones to at least $10$ decimal digits,
although no attempt has been made to establish rigorous error bounds.

We say that a given function $F(z)$ has {\it monotonic zeros} in a region $D$ of the complex plane if, 
when we list the zeros of $F$ in $D$ by increasing real part, the imaginary parts of the zeros are 
monotone nondecreasing.  Formally, if $\{\alpha _1, \alpha _2 , \ldots \}$ are the zeros of $F$ in $D$
numbered so that $\Re (\alpha _i ) < \Re (\alpha _{i+1}$ for $i\ge 1$, then
$\Im (\alpha _i) \le \Im (\alpha _{i+1})$.  (We assume $F$ has at most one zero on the intersection
of any vertical line with $D$.)
The data in the Appendix and other computer runs support 
the following hypothesis.

\begin{Conjecture}
\label{hyper}
For $N \in \mathbb N$, $\Xi _N (z)$ has monotonic zeros in $Q$.
\end{Conjecture}

\begin{Proposition}
\label{RH}
Conjecture \ref{hyper} implies the Riemann Hypothesis.
\end{Proposition}

\begin{proof}
This follows from the argument principle, combined with
the simple fact that a function with infinitely many positive real zeros, and with 
monotonic zeros in $Q$, has only real zeros in $Q$.  
Assume the RH is false, and let $\tau$ be the zero of $\Xi (z)$ in $Q$ with minimal real part, among those zeros with 
positive imaginary part, and let $\tau = \sigma +it$.  By the argument principle, we have
\begin{align}
\frac {1}{2\pi i} \oint_{C_{\epsilon}}  \frac {\Xi ^{\prime} (z)}{\Xi (z)} \, \, dz =1,
\end{align}
where the integral is taken counterclockwise around a circle $C_{\epsilon}$ centered at $\tau$, 
of small radius $\epsilon$, so no other zeros of $\Xi (z)$
are enclosed in $C$.  Next 
choose $N$ sufficiently large so that $\Xi _N (z)$ has a real zero $\gamma$ with $\gamma > 2\sigma$.  We can do this
since $\Xi _N (z)$ converges uniformly to $\Xi (z)$ on compacta in $Q$, 
both $\Xi (z)$ and $\Xi _N (z)$ are real on the real line, and since $\Xi (z)$ has infinitely many positive real zeros.
By assumption $\Xi _N (z)$ has monotonic zeros in $Q$, hence has no non-real zeros in $Q$ with real part
less than $2\sigma$.  This implies 
\begin{align}
\frac {1}{2\pi i} \oint_{C_{\epsilon}}   
\frac {\Xi _N ^{\prime} (z)}{\Xi _N (z)}  dz =0,
\end{align}
and so
\begin{align}
1 &= \frac {1}{2\pi i} \oint_{C_{\epsilon}}  \frac {\Xi ^{\prime} (z)}{\Xi (z)} - 
\frac {\Xi _N ^{\prime} (z)}{\Xi _N (z)}  dz \\ 
\label{key}
 &= \frac {1}{2\pi i} \oint_{C_{\epsilon}}  \frac {\Xi ^{\prime} (z)\Xi _N (z) - \Xi _N ^{\prime} (z) \Xi (z) }
{\Xi (z) \Xi _N (z)}  \, \, dz .
\end{align}
On the closed and bounded set $C_{\epsilon}$, $|\Xi (z)|$ is nonzero and hence must assume an absolute minimum $\delta >0$.
Due to the uniform convergence, as $N \to \infty$, the minimum of $|\Xi _N (z)|$ on $C_{\epsilon}$ must eventually be
greater than $\delta /2$.  Hence for large $N$ the modulus of the denominator of the integrand in
(\ref{key}) is bounded away from zero, but (since $\Xi _N ^{\prime} (z)$ approaches $\Xi ^{\prime} (z)$ uniformly)
the numerator approaches zero, and so the integral will also
approach zero, a contradiction.
\end{proof}

\begin{Remark}
\label{weaker}
A weaker form of Conjecture \ref{hyper}, which still implies RH, is that there are no non-real zeros of $\Xi _N (z)$ in $Q$
whose real part is less than the largest real zero of $\Xi _N (z)$.  Since
$\Xi _N (z)$ is real for real $z$, the real zeros can be found by looking at sign changes along the real line.
Then the argument principle can be used via a numerical integration to obtain the total number of zeros
of $\Xi _N (z)$ with real part not greater than the largest real zero, and matched against the number of real
zeros.  Computations along these lines 
indicate this weaker form of
Conjecture \ref{hyper} is true at least for $N\le 10$.  Below we list the largest real zero and
number of real zeros of $\Xi _N (z)$ for $N \le 10$.
\end{Remark}
\begin{align}
\notag
\begin{matrix}
\underline{N} \quad & \underline{\text{largest real zero}} & \underline{\text{number of real zeros}} \\
1 \quad & 14.0454395788 & 1 \\
2 \quad & 39.5324810798 & 7 \\
3 \quad & 65.0320737720 & 15 \\
4 \quad & 103.3679880094 & 32 \\
5 \quad & 149.0026994921 & 53 \\
6 \quad & 197.9575955732 & 79 \\
7 \quad & 258.5304836632 & 113 \\ 
8 \quad & 327.3794646017 & 155 \\ 
9 \quad & 406.8174206801 & 207 \\
10 \quad & 489.3900649445 & 263 
\end{matrix}
\end{align}

\section{Other Zeta Functions}
\label{otherzeta}

Many other zeta functions which are conjectured to satisfy a Riemann Hypothesis can be approximated
by sums of hyperbolic gamma functions.
Let $\tau (n)$ denote 
Ramanujan's $\tau $-function.  Since the function
\begin{align}
F(z) = \sum_{n=1}^{\infty} \tau (n) \exp (2\pi  i n z)
\end{align}
is a modular form of weight $12$, we have
\begin{align}
F(ix) = x^{-12} F(i/x),
\end{align}
which can be used to show
\begin{align}
\label{Ram}
(2\pi )^{-s} \Gamma (s) \sum_{n=1}^{\infty} \frac {\tau (n)}{n^s} =
\sum_{n=1}^{\infty} \tau (n) \left ( 
\int_{1}^{\infty}  x^{s-1} \exp (-2\pi nx) \, dx  +
\int_{1}^{\infty}  x^{-s-1} x^{12} \exp (-2\pi nx) \, dx \right )
\end{align}
which implies
\begin{align}
\label{ramaj}
(2\pi )^{-6-iz} \Gamma (6+iz) \sum_{n=1}^{\infty} \frac {\tau (n)}{n^{6+iz}} =
\sum_{n=1}^{\infty} \tau (n) 
G(z;2\pi n,6),
\end{align}
again a uniformly convergent sum of hyperbolic gamma functions. (It is known that 
$|\tau (n)| = O(n^6)$ \cite{Apostol}.)   The function defined by the left-hand-side of
(\ref{ramaj}) is known as the Ramanujan $\Xi$-function, which we denote $\Xi_{\Delta} (z)$, and the approximate obtained by
truncating the series on the right-hand-side of (\ref{ramaj}) 
after $N$ steps we denote $\Xi _{\Delta ,N}(z)$.
Ramanujan conjectured that $\Xi_{\Delta} (z)$ 
has only real zeros, which is still open.  We mention that Ki \cite{Ki08} has studied the zeros
of different approximates to the Ramanujan $\Xi$-function.

\begin{Conjecture}
\label{RamDelta}
For $N \in \mathbb N$, $\Xi _{\Delta, N}(z)$ has monotonic zeros in $Q$.
\end{Conjecture}

Computations analogous to those described in Remark
\ref{weaker} indicate that 
the corresponding weaker form of Conjecture \ref{RamDelta}, that there are no nonreal
zeros of 
$\Xi_{\Delta ,N}(z)$ in $Q$ with real part smaller than the largest real zero of
$\Xi_{\Delta ,N}(z)$, is true at least for $N\le 10$.  Below we list the number of
real zeros and the largest real zero of $\Xi _{\Delta,N}(z)$ for $N\le 10$.
\begin{align}
\notag
\begin{matrix}
\underline{N} \quad & \underline{\text{largest real zero}} & \underline{\text{number of real zeros}} \\
1 \quad &  & 0 \\
2 \quad & 9.1937689444922& 1 \\
3 \quad & 13.885647964708& 2 \\
4 \quad & 21.358047646119& 5 \\
5 \quad & 25.047323063922& 6 \\
6 \quad & 28.706422677689& 8 \\
7 \quad & 33.529929734593& 11 \\ 
8 \quad & 36.535376767485& 12 \\ 
9 \quad & 40.190608700694& 14 \\
10 \quad & 44.761812314903& 17
\end{matrix}
\end{align}

More generally, we can start with any entire modular cusp form 
\begin{align}
f(z) = \sum_{n=1}^{\infty} c(n) \exp (2\pi i z)
\end{align}
of weight $2k$, and set
\begin{align}
B(z) = (2\pi)^{-k-iz}\Gamma (k+iz) \sum_{n=1}^{\infty} \frac {c(n)}{n^{k+iz}}.
\end{align}
The modularity of $f$ can be used to show \cite[pp. 137-138]{Apostol} that
\begin{align}
\label{modular}
B(z) &= \int _{1}^{\infty} f(it) \left ( t^{k+iz} +
(-1)^{k} t^{k-iz} \right) dt/t \\
& 
\label{rhmod}
=\sum_{n=1}^{\infty} c(n) \int_{2\pi n} ^{\infty} 
\exp (-u) \left (  (u/2\pi n)^{k+iz} + (-1)^{k} (u/2\pi n)^{k-iz} \right ) \,  
\frac {du}{u},
\end{align}
where we have used the well-known bound $c(n) = O(n^k)$ to justify the interchange
of summation and integration.  Thus we see that, at least for $k$ even,
$B(z)$ is also a uniformly convergent (infinite) linear combination of hyperbolic gamma functions.

We can do a similar calculation for Dirichlet L-series $L(s,\chi)$ with 
$\chi$ a primitive character with modulus $q$.  Assume for the moment that
$\chi (-1)=1$, and define
\begin{align}
\Xi (z,\chi) =  \pi ^{-(1/4+iz/2)} \Gamma (1/4+iz/2) q^{1/4+iz/2} L(1/2+iz,\chi).
\end{align}
From \cite[pp.68-69]{Davenport} we have (using the fact that $\chi (n) = \chi (-n)$)
\begin{align}
\label{dir1}
\Xi (z,\chi) &=  \sum_{n=1}^{\infty} \chi (n) 
\int _{1}^{\infty} \exp (-n^2\pi t/q) \, t^{1/4+iz/2} \, \frac {dt}{t} \\
\label{dir2}
&+ \sqrt{q} w(\chi)
\sum_{n=1}^{\infty} \overline {\chi} (n) 
\int _{1}^{\infty} \exp (-n^2\pi t/q) \, 
t^{1/4-iz/2} \, \frac {dt}{t} \\
&
\label{erhplus1}
=\sum_{n=1}^{\infty} \chi (n) 
\frac {\Gamma (1/4+iz/2,\pi n^2/q)  }
{(\pi n^2/q)^{1/4+iz/2}}  +\\
& 
\notag
\sqrt {q} w(\chi) \sum_{n=1}^{\infty} \overline {\chi} (n) 
\frac {\Gamma (1/4-iz/2,\pi n^2/q)  }
{(\pi n^2/q)^{1/4-iz/2}}, 
\end{align}
where $\sqrt {q} w(\chi)$ is a certain complex number of modulus one and 
$\overline {\chi} (n)=\overline {\chi (n)}$ is the conjugate character.  Furthermore if
$\chi (-1)=-1$ we get
\begin{align}
\label{dir1b}
\pi ^{-(3/4+iz/2)} & \Gamma (3/4+iz/2) q^{3/4+iz/2} L(1/2+iz,\chi) \\
 &=  \sum_{n=1}^{\infty} n\chi (n) 
\int _{1}^{\infty} \exp (-n^2\pi t/q) t^{3/4+iz/2} \, \frac {dt}{t}+ \\
\label{dir2b}
&i\sqrt{q} w(\chi)
\sum_{n=1}^{\infty} n\overline {\chi} (n) 
\int _{1}^{\infty} \exp (-n^2\pi t/q)  
t^{3/4-iz/2}) \, \frac {dt}{t} \\
&
\label{erhminus1}
= \sum_{n=1}^{\infty}  n\chi (n) 
 \frac {\Gamma (3/4+iz/2,\pi n^2/q)  }
{(\pi n^2/q)^{3/4+iz/2}}  +\\
&
\notag
i\sqrt{q} w(\chi)
\sum_{n=1}^{\infty} n\overline {\chi} (n)
\frac {\Gamma (3/4-iz/2,\pi n^2/q)  }
{(\pi n^2/q)^{3/4-iz/2}}, 
\end{align}
where again $|i\sqrt {q}w(\chi)|=1$.  

Let $F(z)$ be the function defined by truncating the series on the right-hand-side of 
(\ref{rhmod}), or both of the series on the right-hand-side of
either (\ref{erhplus1}) or (\ref{erhminus1}), after $N$ steps.  
If Conjectures and \ref{hyper} and \ref{RamDelta} are true, one might suspect that $F(z)$ 
also has monotonic zeros in $Q$, although the author has not yet done any computations with these more
general sums.

Another interesting question is where the zeros of $\Gamma (z,a)$ are, for $a$ a positive real
number.  
Neilsen \cite{Neilsen} showed that $\Gamma (z,a)$ has no zeros in $\Re (z) <a$, and 
Gronwall \cite{Gro16}
proved that $\Gamma (z,a)$ has infinitely many zeros in $Q$.  Mahler \cite{Mah30}
showed that, as $a \to \infty$, the zeros of $\Gamma (az,a)$ cluster about
the limiting curve 
\begin{align}
\Re \left (z \log z + 1-z \right ) =0.
\end{align}
Tricomi and other authors
have investigated the zeros of (the meromorphic continuation of)
$\Gamma (z,a)$ as a function of $a$, for fixed $z$.
In summary, not much information 
seems to be known about the zeros of $\Gamma (z,a)$, for $a$ a fixed positive real number (
although the literature contains a number of 
detailed results on the zeros of the lower
incomplete gamma function $\Gamma (z) - \Gamma (z,a)$).  
In $1998$ Gautschi \cite{Gau98} published a nice survey of 
known results on incomplete gamma functions. 

Computer calculations support the following.
 
\begin{Conjecture}
\label{wgamma}
For any fixed positive real number $a$, the incomplete gamma function $\Gamma (z,a)$ has monotonic zeros in $Q$.
\end{Conjecture}
Although some analog of Conjecture \ref{wgamma} may be true 
for hyperbolic gamma
functions,  in Section \ref{combos} 
we show that there exist some choices of $a,b \in \mathbb R$, $a>0$ for
which $G(z;a,b)$ does not have monotonic zeros in $Q$.

\begin{Remark}
To say a function has monotonic zeros in $Q$ is equivalent to saying that the first differences of
the imaginary parts of the zeros are all nonnegative.  The zeros in $Q$
of $\Gamma (z,a), \, \, a>0$ 
seem to satisfy the more general property 
that the $k$-th differences of the imaginary parts of the zeros are 
positive for $k$ odd and negative for $k$ even, for $k \le 7$ or $8$.  Thus  
these zero sets seem to have extra structure beyond being monotonic.
For linear combinations of hyperbolic gamma functions
the same phenomena seems to occur for $x$ sufficiently large, 
which may be due to the main term in the asymptotics 
controlling the zeros.
\end{Remark}

\section{Asymptotics}
\label{asymptotics}

Throughout this section $z=x+iy \in Q$, $x,y\ge 0$, $\theta = \arg (z), 0\le \theta \le \pi/2$,
$a,b \in \mathbb R, a>0$.
We begin with Stirling's formula and some other known results:
\begin{align}
\label{Stirling}
\Gamma (z)= \sqrt{\frac{2\pi}{z}}\left ( \frac{z}{e} \right ) ^{z} \left ( 1 + O(1/z)
\right )
\end{align}
(where $ -\pi < \arg (z) \le \pi$)
\begin{align}
\label{ig}
\Gamma (z,a) = \Gamma (z) - \frac{a^z}{ze^a} 
\left ( 1 + \frac{a}{z+1} + \frac{a^2}{(z+1)(z+2)}
+\ldots + \frac{a^k}{(z+1)_k} + \ldots \right )
\end{align}
(\cite[Vol. II, p. 135]{Bateman})
\begin{align}
\label{simple}
\frac{\Gamma (b+z)}{\Gamma (z)} = z^b\left ( 1 + O(1/z) \right)
\end{align}
(\cite[Vol. I, p. 47]{Bateman}),
where in the big-Oh results we mean as $|z| \to \infty$.  From (\ref{Stirling}) we see
that 
\begin{align}
|\Gamma (z)| = \sqrt {\frac{2\pi}{|z|}} \exp ( x\ln |z| -x - y\theta) 
\left(1 + O(1/|z|) \right ).
\end{align}
Now $\Gamma (z,a)=0$ if and only if $\Gamma (z,a)/\Gamma (z)=0$ so by (\ref{ig}) 
$\Gamma (z,a)=0$, $|z|$ large implies
\begin{align}
\frac{a^z}{z\Gamma (z) e^a} \sim 1
\end{align}
or
\begin{align}
\exp(x\ln |z| -x -y\theta) \sqrt{2\pi|z|} \sim \exp(x\ln a -a).
\end{align}
If $y$ remains bounded and $x\to \infty$, the left hand side above grows too fast, while
if $x$ remains bounded and $y \to \infty$ it decreases too fast.  Hence we need both 
$x,y \to \infty$.  Furthermore, $y\theta$ must be asymptotic to $x\ln x$, 
hence $\theta \to \pi/2$ and
the zeros of $\Gamma (z,a)$ in $Q$ satisfy 
$y\sim \frac {2}{\pi} x\ln x$ as $|z| \to \infty$.

To perform the same analysis for $G(z;a,b)$, first note that
\begin{align}
|\Gamma (iz)| &= \sqrt{ \frac{2\pi}{|z|} }
\exp \left ( -y\ln |z| +y -x(\theta + \pi/2) \right ) \left ( 1 + O(1/|z|) \right ) \\ 
|\Gamma (-iz)| &= \sqrt{ \frac{2\pi}{|z|} }
\exp \left ( y\ln |z| -y +x(\theta - \pi/2) \right ) \left ( 1 + O(1/|z|) \right ). 
\end{align}
Thus using (\ref{simple}),
\begin{align}
\label{G1}
\frac{|\Gamma (b+iz)|}{|a^{b+iz}|} &= \sqrt{ \frac{2\pi}{|z|} } |z|^b
\exp \left ( -y\ln |z| +y -x(\theta + \pi/2) -(b-y)\ln a \right ) \left ( 1 + O(1/|z|) 
\right ) \\
\label{G2}
\frac{|\Gamma (b-iz)|}{|a^{b-iz}|} &= \sqrt{ \frac{2\pi}{|z|} } |z|^b
\exp \left ( y\ln |z| -y +x(\theta - \pi/2) -(b+y)\ln a \right ) \left ( 1 + O(1/|z|) \right ). 
\end{align}
From (\ref{ig}),
\begin{align}
\label{GG1}
G(z;a,b) = 
\frac{\Gamma (b+iz)}{a^{b+iz}} +
\frac{\Gamma (b-iz)}{a^{b-iz}}  \\
\notag
-\frac{1}{(b+iz)e^a} \sum_{k=0}^{\infty} \frac{a^k}{(b+iz+1)_k}    
-\frac{1}{(b-iz)e^a} \sum_{k=0}^{\infty} \frac {a^k}{(b-iz+1)_k}.
\end{align}
One finds  
\begin{align}
\label{GG3}
-\frac{1}{(b+iz)e^a} \sum_{k=0}^{\infty} \frac{a^k}{(b+iz+1)_k} 
-\frac{1}{(b-iz)e^a} \sum_{k=0}^{\infty} \frac {a^k}{(b-iz+1)_k} \\ 
\notag
= \frac {2(a-b)}{e^a z^2} + \frac{2(b^3-2ab^2-ab+3a2b+3a^2-a^3)}{e^az^4} + O(1/z^6).
\end{align}
If $y$ remains bounded and $x$ doesn't, then the first two terms on the right-hand-side of (\ref{GG1}) 
approach $0$ like $\exp(-x\pi /2)$, so by (\ref{GG3}), the expansion of
$G(z;a,b)$ in negative powers of $z$ has a nonzero coefficient of $z^{-2}$, unless $a=b$ in which case
it has a nonzero coefficient of $z^{-4}$.  In either case it cannot equal $0$ for sufficiently
large $x$.
If $x$ remains bounded and $y$ doesn't, then the second term on the right-hand-side of
(\ref{GG1}) blows up, while the others approach $0$.  
So as $|z|\to \infty$, for fixed $a,b$, if $G(z;a,b)=0$ we need both $x,y \to \infty$, and 
thus from (\ref{G2}) $x(\theta - \pi/2)+y\ln |z|$ cannot approach positive or negative infinity
too quickly.  In fact 
the zeros of $G(z;a,b)$ must satisfy 
$x \sim  \frac {2}{\pi} y \ln y$ as $|z| \to \infty$.

The argument above also applies to any function of the form
\begin{align}
\label{GGH}
\sum_{k=1}^N u_k\, G(z;a_k,b_k)
\end{align}
$a_k,b_k \in \mathbb R, u_k \in \mathbb C, a_k>0 $, i.e. any $\mathbb C$-linear
combination of hyperbolic gamma functions.  For if you have a linear combination of terms like
(\ref{GG1}), more than one of which is approaching $\infty$, the linear combination must also approach
$\infty$, since by taking into account the contribution of $a_k,b_k$, no two such terms can approach
$\infty$ at the same rate.  The other parts of the argument follow through similarly
(the coefficient of $1/z^{2j}$ in the appropriate
version of (\ref{GG3}) must be nonzero for some $j$, else (\ref{GGH}) would essentially reduce to
a linear combination of Gamma functions, which would not be entire) 
and thus the zeros of any function of the form (\ref{GGH}) also satisfy
$x \sim  \frac{2}{\pi} y \ln y$ as $|z| \to \infty$.

\section{Linear Combinations}
\label{combos}

The examples in Section \ref{otherzeta} involving Dirichlet characters motivate
defining, for $a,b \in \mathbb R$, $a>0$, $w, \alpha \in \mathbb C$, $|w|=1$, 
the generalized hyperbolic 
gamma function as 
\begin{align}
G(z;a,b,\alpha ,w) = 
\alpha \frac {\Gamma (b+iz,a)}{a^{b+iz}} +
\overline {\alpha}   \frac {\Gamma (b-iz,a)}{a^{b-iz}}w.
\end{align}
The author has made several hundred computer runs, calculating the zeros of various
arbitrary $\mathbb C$-linear combinations of generalized hyperbolic gamma functions in
different regions of $Q$.  Surprisingly, in
all of these runs the zeros turned out to be monotonic.  Perhaps this results from a 
mysterious analytic principle, not yet understood, which causes generic sums of generalized
hyperbolic gamma functions to have a high probability of having monotonic zeros.  In this 
case the RH could be a consequence of this analytic principle, combined with Hardy's 
theorem that $\Xi (z)$ has infinitely many real zeros.

It is not the case though that all linear combinations of generalized hyperbolic gamma functions
have monotonic zeros in $Q$.  For example, consider any function of the form 
$t_1A(z) + t_2 B(z)$, where
$t_1,t_2$ are positive real numbers and $A(z),B(z)$ are any two $\Xi$-functions,
corresponding to different
zeta functions 
from Section \ref{otherzeta}, which are real on the real line.  Since $A(z)$ and $B(z)$ have infinitely many real zeros, 
as $z \to \infty$ along the real line they each oscillate from positive to negative.  
Unless there is an
unsuspected correlation between
the two, any real linear combination of them will also oscillate and thus have 
infinitely many real zeros.  But only for very special choices
of $A,B,t_1,t_2$ will this linear combination correspond to a zeta function with an Euler product, and without an underlying Euler product it is 
generally expected that such functions 
will have infinitely
many non-real zeros as well.  In particular, $\Xi (2z) + \Xi _{\Delta ,5}(z)$, the sum of the 
Riemann $\Xi$-function and the fifth approximate to the 
Ramanujan $\Xi$ function, has a non-real zero with real part between 19 and 23, and a real zero at
$z=24.99871$.  (Here $\Xi (z)$ is evaluated at $2z$ to make the two functions compatible,
since the expression (\ref{XiG}) of $\Xi (z)$ in terms of
hyperbolic gamma functions involves $z/2$, while that of $\Xi _{\Delta}$ involves $z$.) 

More simply, one can create an example of non-monotonic zeros
by considering what happens to $\Xi _N (z)$ as $z \to \infty$ along the positive real line.
From (\ref{GG1}) and (\ref{GG3}) we get
\begin{align}
\label{GX1}
G(x;a,b) = \frac {2(a-b)}{e^a} \frac {1}{x^2} + O \left ( \frac {1}{x^4}\right ).
\end{align}
Applying this to (\ref{GXi}) yields 
\begin{align}
\Phi _n (x) = 
\frac {4 n^4\pi ^2 (n^2 \pi - 9/4) - 6n^2\pi ( n^2\pi - 5/4)}{\exp(n^2\pi) } \frac {1}{x^2}
+ O \left (\frac {1}{x^4} \right ),
\end{align}
which shows
\begin{align}
\Phi _1 (x) &= -.01974938206/x^2 + O(1/x^4) \\
\notag
\Phi _2 (x) &= .01974934121/x^2 + O(1/x^4) \\
\notag
\Phi _3 (x) &= .4132639753 \times 10^{-7}/x^2 + O(1/x^4).
\end{align}
For $k>3$ the coefficient of $1/x^2$ in $\Phi _k (x)$ is positive.
Thus the coefficient of $1/x^2$ in 
$\Xi _N (x)$ is positive for $k\ge 3$ and negative for $1\le k<3$.
It follows that $t\Phi_1(x)+\Phi_2(x)$ is positive for sufficiently large $x$ if $t=0$ 
and negative for sufficiently large $x$ if $t=1$, and so $t\Phi_ 1(z)+\Phi _2(z)$ has a zero on the 
real line for some large $x$ and some $0<t<1$.  One finds in fact that there is a real zero
between $x=10^5$ and $x=10^6$ when $t=.999997907459$, 
which occurs after many non-real zeros, giving an example
of non-monotonic zeros.  This real zero travels
travels quickly left as $t$ increases through real values, 
arriving at $x=39.53248$ (the largest real zero of $\Xi _2(z)$) when $t=1$.

One could hope that any series of the form
\begin{align}
F(z) = \sum_{k=1}^{\infty} c_k \Phi _k(z), \quad c_k \in \mathbb R, c_k \ge 0
\end{align}
has monotonic zeros, modulo the problem of the forced real zero described above, which could be avoided by
say requiring $c_1=c_2=c_3$, since then 
\begin{align}
F(x) = \frac {c}{x^2} + O \left ( \frac {1}{x^4} \right )
\end{align}
for some $c>0$. 
A family $\{f_1(z),\ldots ,f_k(z)\}$ of polynomials with real coefficients 
is called {\it compatible} if 
\begin{align}
\sum_{j=1}^k c_j f_j (z)
\end{align}
has only real zeros whenever $c_j \in \mathbb R , c_j \ge 0$ for $1\le j \le k$.  It is
called pairwise compatible if $\{f_i,f_j\}$ forms a compatible family for each pair
$1\le i < j \le k$.  Chudnovsky and
Seymour \cite{ChSe07} have shown that a family of polynomials 
whose members have only real zeros and
positive leading coefficients is compatible if and only if it is pairwise compatible.  
Does a similar
statement hold if we replace real polynomials with positive leading
coefficients and only real zeros, by even, real, entire functions
having monotonic zeros in $Q$?

Eq. (\ref{GX1}) also leads to an 
example of a single hyperbolic gamma function with non-monotonic zeros.
For very large $x$ clearly $G(x;a,2a)<0$ and 
$G(x;a,0)>0$, so there must be some value of $b$, $0<b<2a$ for which 
$G(x;a,b,1,1)=0$ for some very large
$x$, which will thus have non-monotonic zeros in $Q$.  

Another interesting phenomena occurs when we consider the zeros of
$\Xi _k(z) + t\, \Phi _{k+1}(z)$.  As we let $t$ vary continuously from $0$ to $1$,
computations indicate that the imaginary parts of the non-real zeros decrease
monotonically (i.e. continuously), in a very regular manner, until, for high enough $k$, they collide with 
the corresponding zero (from Schwartz reflection) in the fourth quadrant, and arrive on
the real line, where they remain.
\begin{Conjecture} 
\label{tparameter}
For $k\ge 1$, 
the imaginary part of each non-real zero of
$\Xi _k(z) +t\, \Phi _{k+1}(z)$ decreases monotonically (i.e. continuously) as $t$ goes from $0$ to $1$.
\end{Conjecture}

\section{The modulus on vertical rays}

It is known that the RH is equivalent to the statement that the modulus of
$\Xi (z)$ is monotone increasing along any vertical ray which starts at a point $x\ge 0$ on the
nonnegative real line and travels straight upward to $x + i\infty$.  (Clearly if the RH is false
the statement is false.  On the other hand, if the RH is true, start with Hadamard's 
factorization theorem
\begin{align}
|\Xi (z)| = |\Xi (0)| \prod _i |1 - z^2/\alpha _i ^2|,
\end{align}
where the $\alpha _i$ are the positive real zeros of $\Xi (z)$, and take the partial derivative
with respect to $y$, where $z=x+iy$.  This is easily seen to be positive for positive $y$.)

For a given function $F(z)$, analytic in $Q$, and $\alpha$ a nonnegative real number, let
$M(F,\alpha)$ denote the value of $y\ge 0$ where the function $|F(\alpha+iy)|$ is minimal, i.e. the
height where the minimum of the modulus of $F$ occurs on the vertical ray starting at
$\alpha$ and going straight up.  (If this minimum occurs at more than one $y$, let $M(F,\alpha)$
denote the lim inf of such $y$.)  Finally call the set of pairs $\{(\alpha,M(F,\alpha))\}$ in $Q$ the
``M-curve" of $F$.  

Examples of these curves are given in Figures \ref{minmod1}, 
\ref{minmod2},  and 
\ref{minmod3} 
below.  To calculate $M(F,\alpha)$ for a given pair $F$ and $\alpha$ 
the author
simply calculated the modulus of $F(\alpha+iy)$ at many closely spaced grid points $y$ 
and then chose the $y$ which
gave the minimum of these numbers (from the asymptotics, the modulus of a sum of
hyperbolic gamma functions increases quite rapidly when $y$ increases beyond a certain point).   
This same procedure
was followed for several closely spaced grid points $\alpha$, and the pairs $(\alpha,M(F,\alpha))$ 
then plotted
on a grid using Maple, with the result looking something like a continuous curve.
They seem to have the (as yet unexplained) property that
the non-real zeros of $F$ in $Q$ occur at the same places as 
the local maxima of the
M-curve.  
If so, this would give another (rather informal)
method for calculating the zeros of $\Xi _N (z)$ and other sums of hyperbolic gamma functions
which doesn't depend on the argument principle or
Newton's method.   Other computations indicate this property may also hold for any real polynomial with no two zeros 
in the upper half-plane with the same real part. 

\subsection{Some Notes on the Computations}

The list of zeros of $\Xi _1 (z)$ in the Appendix was calculated using the argument principle and Newton's method
by a call to the function analytic in Maple, as were the zeros of $\Xi _{\Delta ,1}(z)$,
$\Xi _{\Delta ,2}(z)$, 
and the zeros of the sum of two generalized hyperbolic gamma functions.  This procedure
failed however to compute the zeros of $\Xi _2 (z)$, as the program never finished even after running for over two 
days on the Sun system at the University of Pennsylvania.  Note that $\Xi _k(z)$ is a sum of $4k$ incomplete gamma
functions, which may explain why the computation of the zeros of $\Xi _k (z)$ quickly becomes difficult.  The author
was able to find the zeros of $\Xi _2(z)$ by starting with the zeros of $\Xi _1 (z)$ and using Newton's method to
find the zeros of 
$\Xi _1 (z) + t \, \Phi _2 (z)$, for $t$ a small positive number, then recursively using these 
new zeros and Newton's method
to find the zeros for a slightly larger value of $t$, slowly increasing $t$ until $t$ equaled $1$.  
The author then tried to
compute the zeros of $\Xi _3 (z)$ in the same way, by starting with the zeros of
$\Xi _2 (z)$ and using Newton's method to compute the zeros of $\Xi _2 (z) + t \, \Phi _3 (z)$ for small $t$, and
gradually increasing $t$ as before.  This worked well until $t$ became very close to $1$, about $t=.99$, at which
point Newton's method no longer converged.  If the property
discussed in the previous paragraph holds though, we can see from Figure \ref{minmod2} where the non-real zeros of
$\Xi _3(z)$ in the range $0 \le \Re (z) \le 120$ are.

\begin{figure}
\includegraphics[scale=.8]{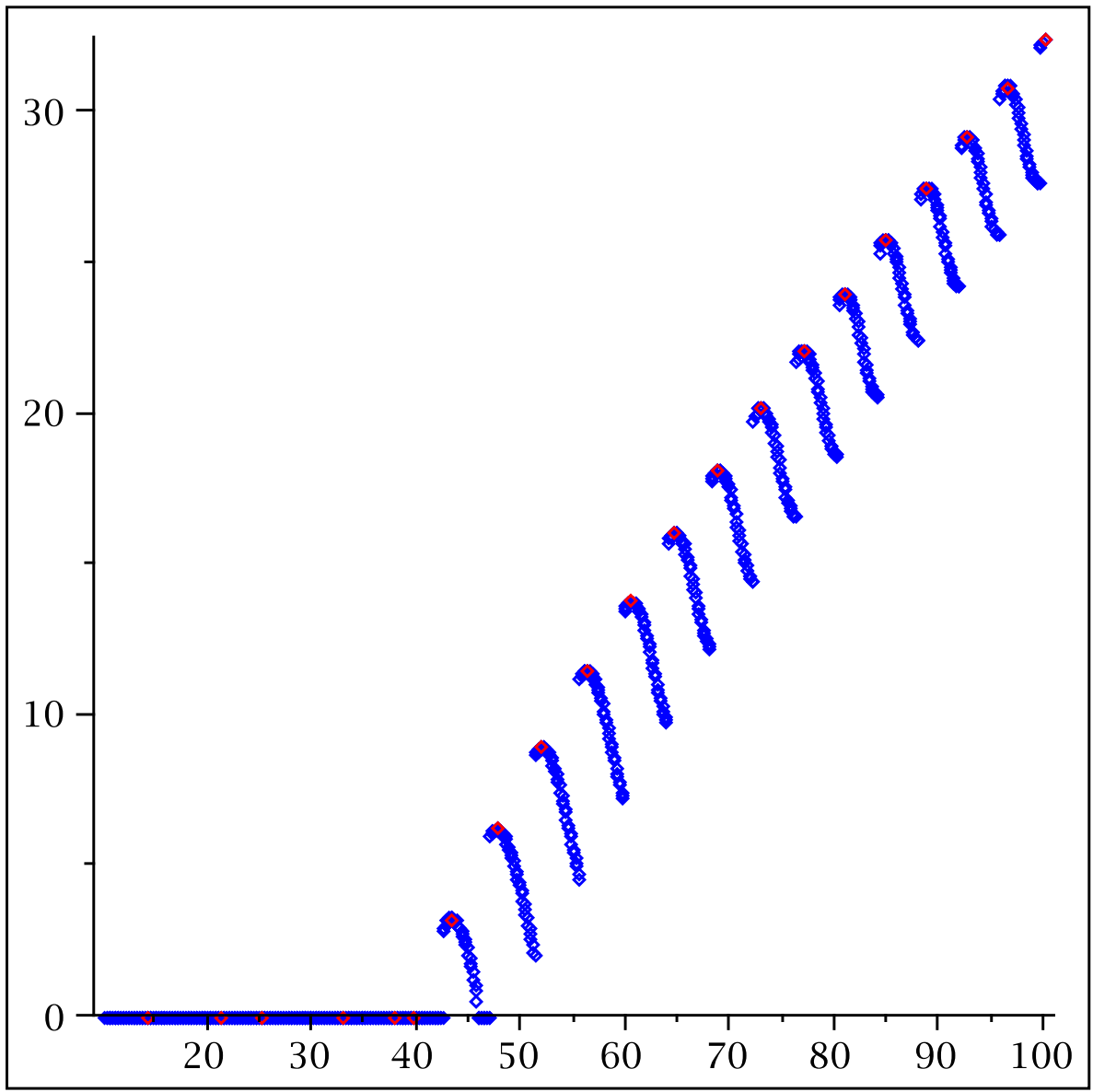}
\caption{The M-curve (in blue) for $\Xi _2 (z)$, $0 \le \Re (z) \le 100$.
The zeros are overlaid in red;
note they occur at the local maxima of the M-curve.}
\label{minmod1}
\end{figure}

\begin{figure}
\includegraphics[scale=.8]{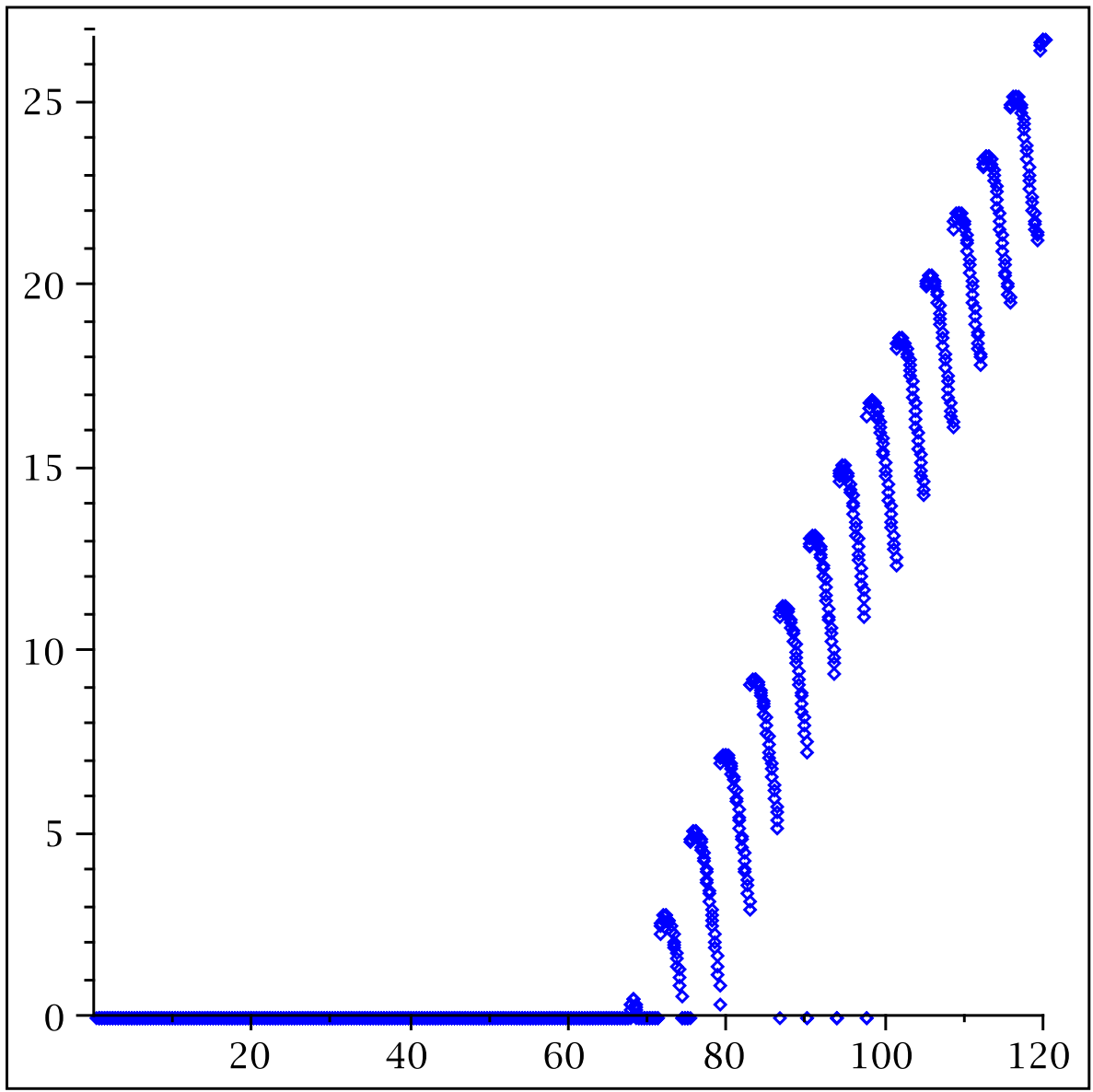}
\caption{The M-curve for $\Xi _3 (z)$, $0 \le \Re (z) \le 120$.}
\label{minmod2}
\end{figure}

\begin{figure}
\includegraphics[scale=.8]{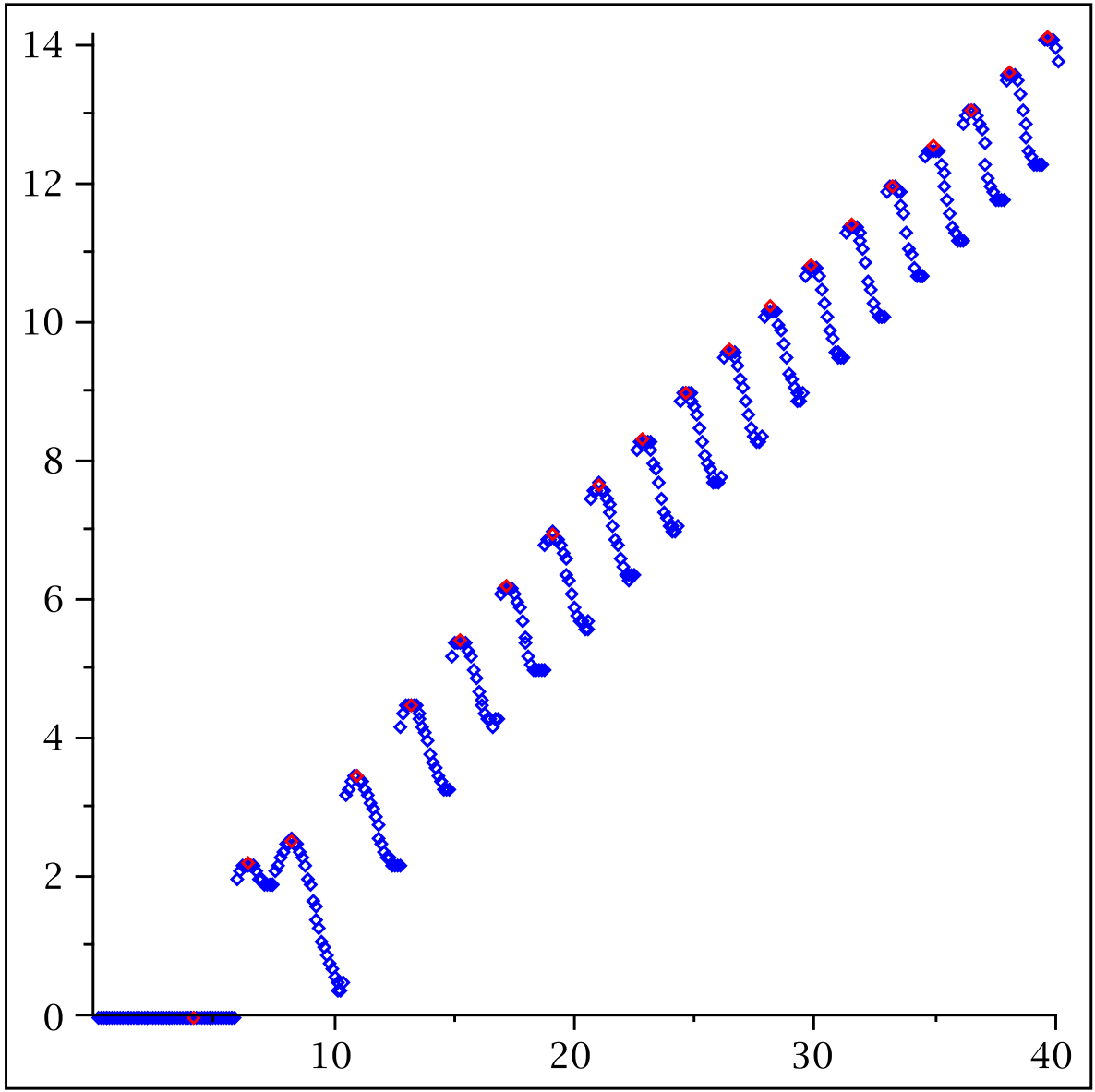}
\caption{The M-curve for a sum of $8$ arbitrarily chosen
generalized hyperbolic gamma functions.  The zeros (of the sum of these $8$)
 are overlaid in red.}
\label{minmod3}
\end{figure}

\newpage

\section{Appendix}
\label{Appendix}

$Digits:=25$, zeros of first $\Xi$-approximate $\Xi_1(z)$ in rectangle with corners $(0,-.1),(100,500)$.
\medskip

    14.04543957882981756479858,

    20.62534600592171760132974 + 2.697151842339519632505712 I 

    26.05616693357829946749575 + 7.125359707612690330897455 I 

    31.50143137824977099308422 + 10.72915037105496782822450 I 

    36.72702276874255239918647 + 13.75961410603683555833019 I 

    41.73703479849622101486046 + 16.44012737324329251859479 I 

    46.56622866997881255099908 + 18.88186965378958902053812 I 

    51.24456582311629453468990 + 21.14750420601374895347492 I 

    55.79525368022472028456165 + 23.27625685820891335493023 I 

    60.23621426525993802296865 + 25.29458549895993860216014 I 

    64.58150497097301796850798 + 27.22133555778112035831075 I 

    68.84235653395121330563843 + 29.07049609150585601287785 I 

    73.02789933182939276748060 + 30.85279227139366634464017 I 

    77.14567324003250763696303 + 32.57666324204392832752644 I 

    81.20199121212953110713480 + 34.24889253114939152723783 I 

    85.20220345212231662722890 + 35.87503096670553315342957 I 

    89.15089297349449318064800 + 37.45968995259880236581690 I 

    93.05202292717284600209187 + 39.00675061925213970478000 I 

    96.90904939663401491219210 + 40.51951660155401879741380 I

\medskip

                               1st differences of imaginary parts of zeros
\medskip

                           2.697151842339519632505712

                           4.428207865273170698391743

                           3.603790663442277497327045

                           3.03046373498186773010569

                           2.68051326720645696026460

                           2.44174228054629650194333

                           2.26563455222415993293680

                           2.12875265219516440145531

                           2.01832864075102524722991

                           1.92675005882118175615061

                           1.84916053372473565456710

                           1.78229617988781033176232

                           1.72387097065026198288627

                           1.67222928910546319971139

                           1.62613843555614162619174

                           1.58465898589326921238733

                           1.54706066665333733896310

                           1.51276598230187909263380

\medskip

$Digits:=35$, zeros of $\Xi_2(z)$ in rectangle with corners $(0,-.1),(102,100)$.
\medskip

14.1347251016150223590867934323428320

21.0220425550989420016995644399118565

25.0108186655252580507941448128443705

30.4267684045343046403275534114141879

32.9244008910878391337889771099093884

37.8603410698702633339745855804483258

39.5324810797879508571153350051479630

43.1389080680988950956236929709951507 + 3.28097100306881350163884998294539870 I

47.5227563037011494043923449534796642 + 6.25250911174970354129046143001000600 I

51.8283147335875594556295444210165891 + 8.95857384762622450769949019637180856 I

56.1120029199774070873341435705058778 + 11.4796085175258061625014962442075566 I 

60.3502919249623580499827793888173447 + 13.8416168548275384327656057308155260 I

64.5390675686875954205140494065924195 + 16.0705629583418306170962131727661418 I

68.6761126675575950561373116041091163 + 18.1868197094425056650011185673044960 I

72.7617836933667997393215687715081074 + 20.2066697023948159796720721861483021 I 

76.7975312791721943222372025862518267 + 22.1432248556299027619859319905393103 I

80.7854051988095067379835000872591797 + 24.0071081808438939881307280663316996 I

84.7277019646143359935370107950650673 + 25.8070157111458796320542608491379890 I

88.6267772741056992788936274593068783 + 27.5501321650503597200806915781312903 I 

92.4849402660525200867715711075727066 + 29.2424518834035459512085274002314648 I

96.3043979005927443711172685790765488 + 30.8890225459009695084228517057586202 I 

100.087228624602886284768301012877784 + 32.4941324878154139135000560357642324 I

\medskip

 $Digits:=35$, zeros of $\Xi_3(z)$ in rectangle with corners $(0,-.1),(66,67)$.
\medskip

14.13472514173469376946955013374227983392383419451107438

21.02203963877155659023693964971327708279243073624430253

25.01085758014566673273814860967472222032345345332059477

30.42487612586063748020316156144982639562443306193294638

32.93506158773265354499645423632051649764195395468931578

37.58617815896445021695716617231461056957864550061946742

40.91871901006637538924149719014594063254432154184503222

43.32707329088964789958258005109768490700030014005496697

48.00515051035397786165873118175586051353275737166717210

49.77383399904302558365984997979234680103752864503720614

52.97031190983202078103080161894128897412215226172440818

56.44637829985894482385431963634767652784719134486888209

59.34511184966886217051112646891854833040742035322093409

60.83672336038805375320166674337805885068490701004798353

65.03207377198913910137679883482704065081967504834437506

\medskip
$Digits:=35$, zeros of first approximate $\Xi_{\Delta ,1}$ to the
Ramanujan $\Xi$-function in rectangle with corners $(0,-.1),(100,100)$.
\medskip

9.5489635091412543125672429455388951 + 1.7119172216212706912355445109132207 I

13.789433051893611052566335275629933 + 5.2848734079809646491896114233647142 I

17.427358570320731571246782104695096 + 7.7540563400814520087583759352815352 I

20.757098189027862070211629799284288 + 9.7974743824676388131629735365802915 I

23.879056232416349736892490077064465 + 11.592456287772395127573044664358961 I

26.847356442625470267126707245453301 + 13.219199923922411333605849186386965 I

29.695457505212671889597523331405240 + 14.722240962151850502838429485181816 I

32.445858241734786458856784110868772 + 16.129361168830359821304709531815352 I

35.114596915766277030958889001310311 + 17.459270918804088444729979963005156 I

37.713610437352047890893961093475018 + 18.725280350575037459960230440611499 I

40.252083817737620602530239123250025 + 19.937256671472885578681886805140642 I

42.737274291251073747956340698712113 + 21.102754510914081337479910411119165 I

45.175041128226087752248534272004155 + 22.227709943351466358153291572887016 I

47.570200654826011921382328419781717 + 23.316887844914435394477534028073124 I

49.926772425743961667810425669964510 + 24.374181909750589494277749519061023 I

52.248154899971609917404551354946461 + 25.402822596146688919661995414148711 I

54.537253919041181821883861911330465 + 26.405525330914909309258552834547793 I

56.796578677359823555284912818063663 + 27.384598689419083619944995509463388 I

59.028314745351122620881979673545968 + 28.342025013215928897705357026998109 I

61.234380542300284878652304849348298 + 29.279521587539252797178004730959225 I

63.416471643946175987038219068629433 + 30.198587816232713426722974287579177 I

65.576095995892802432660080685311982 + 31.100542121402673258701892990637098 I

67.714602225166643613493231218811692 + 31.986551176775085187674532167087986 I

69.833202642000750529879098123320031 + 32.857653335480587110184136136077900 I

71.932992106046633644353345531169636 + 33.714777601864923783681632805009935 I

74.014963635252309537091144416299462 + 34.558759141214355169362183745447748 I

76.080021422732136425762615907263059 + 35.390352069528146990609600601997921 I

78.128991771589170231983397657306861 + 36.210240084528180967081058119582790 I

80.162632342789627094434403795275856 + 37.019045367211828466413017905373453 I

82.181640025246587241041487951836980 + 37.817336085881271861478883874617446 I

84.186657672248860977039557949370316 + 38.605632761831562562091441356062965 I

86.178279898670686132064566105956960 + 39.384413700925457309936759214348925 I

88.157058095044852280187866298051226 + 40.154119653347769812443058137516104 I

90.123504784721498632008498428129335 + 40.915157831527436633928427478875093 I

92.078097426893282987651528394221783 + 41.667905391109706656241503132525470 I

94.021281749721926032933713327889188 + 42.412712460188349826344885668577328 I

95.953474683021781433616191339406534 + 43.149904786473380482490406058490112 I

97.875066948097497309259057547588008 + 43.879786059713960896674731680251557 I

99.786425352756449975929052839495573 + 44.602639956801403127156000977181747 I

\medskip
$Digits:=25$, zeros of second approximate $\Xi_{\Delta ,N}$ to the
Ramanujan $\Xi$-function in rectangle with corners $(0,-.1),(100,200)$.
\medskip

             9.19376894449217893

             13.8348516743546257 + 1.54300060155520050 I

             17.0006907368319637 + 4.58047657550007841 I

             20.1893418670461022 + 6.99896197324829350 I

             23.2665587059084961 + 9.04437482423736081 I

             26.2157755564237380 + 10.8464634984827865 I

             29.0559366154717340 + 12.4806210190923921 I

             31.8034718883051467 + 13.9910040730670684 I

             34.4718383897903041 + 15.4050720323654861 I

             37.0718504676122273 + 16.7414677662025767 I

             39.6120959833238977 + 18.0135340186718953 I

             42.0994917987503797 + 19.2311725419877171 I

             44.5397026159453686 + 20.4019721238147549 I

             46.9374285351558464 + 21.5319037938006434 I

             49.2966167316589794 + 22.6257645644359874 I

             51.6206201858570109 + 23.6874764176583670 I

             53.9123161924908609 + 24.7202947836257462 I

             56.1741960636482404 + 25.7269571656683657 I

             58.4084343572848694 + 26.7097915965161903 I

             60.6169430187982438 + 27.6707975724897949 I

             62.8014141732518389 + 28.6117075511920988 I

             64.9633543292524033 + 29.5340343986557084 I

             67.1041120187127688 + 30.4391085162350689 I

             69.2249003477575950 + 31.3281072698702158 I

             71.3268155548021774 + 32.2020785869307989 I

             73.4108524066057312 + 33.0619600720511948 I

             75.4779170679704262 + 33.9085946393025650 I

             77.5288379348185691 + 34.7427434066455078 I

             79.5643748117146020 + 35.5650964167873912 I

             81.5852267334967588 + 36.3762816159807758 I

             83.5920386687224333 + 37.1768724246086591 I

             85.5854072948756678 + 37.9673941604083026 I

             87.5658859982298918 + 38.7483295199556847 I

             89.5339892223226212 + 39.5201232818472100 I

             91.4901962662154965 + 40.2831863625189748 I

             93.4349546156253617 + 41.0378993303884662 I

             95.3686828755524769 + 41.7846154642047194 I

             97.2917733614022269 + 42.5236634258506564 I

             99.2045943961865493 + 43.2553496053959862 I

\medskip

 $Digits:=25$, 
zeros, in rectangle with corners $(10000.0,500),(10025.8,10000)$,
of linear combination of two generalized hyperbolic gamma functions
\begin{align}
\sum_{k=1}^2 \beta _k G(z;A_k,B_k,\alpha _k,w)
\end{align}
where
\begin{align}
\beta _1 &= 1.0+.3i, \, \beta _2=-3-i, 
\alpha _1 = 1.0+.3i, \, \alpha _2=-3-i, \\
w&=0.7648421872844884262558600 + 0.6442176872376910536726144 I \\
B_1&=1.0, \, B_2 = 2.2, \, A_1 = \pi , A_2 = 7.853981633974483096156608
\end{align}

    10000.08899697917243762827 + 1943.699378898319010297001 I,

    10000.84424695361683674126 + 1943.828178641310676925944 I,

    10001.59949024766823939791 + 1943.956976210289663002380 I,

    10002.35472686200573704240 + 1944.085771605493580928470 I,

    10003.10995679730829337130 + 1944.214564827159996509445 I,

    10003.86518005425474436858 + 1944.343355875526428966768 I,

    10004.62039663352379834061 + 1944.472144750830350951281 I,

    10005.37560653579403595119 + 1944.600931453309188556355 I,

    10006.13080976174391025650 + 1944.729715983200321331031 I,

    10006.88600631205174674006 + 1944.858498340741082293155 I,

    10007.64119618739574334772 + 1944.987278526168757942515 I,

    10008.39637938845397052261 + 1945.116056539720588273964 I,

    10009.15155591590437124002 + 1945.244832381633766790543 I,

    10009.90672577042476104237 + 1945.373606052145440516603 I,

    10010.66188895269282807412 + 1945.502377551492710010906 I,

    10011.41704546338613311658 + 1945.631146879912629379767 I,

    10012.17219530318210962295 + 1945.759914037642206290096 I,

    10012.92733847275806375302 + 1945.888679024918401982558 I,

    10013.68247497279117440814 + 1946.017441841978131284633 I,

    10014.43760480395849326601 + 1946.146202489058262623716 I,

    10015.19272796693694481556 + 1946.274960966395618040195 I,

    10015.94784446240332639171 + 1946.403717274226973200536 I,

    10016.70295429103430821021 + 1946.532471412789057410354 I,

    10017.45805745350643340246 + 1946.661223382318553627526 I,

    10018.21315395049611805021 + 1946.789973183052098475038 I,

    10018.96824378267965122043 + 1946.918720815226282254485 I,

    10019.72332695073319499999 + 1947.047466279077648958683 I,

    10020.47840345533278453047 + 1947.176209574842696284945 I,

    10021.23347329715432804280 + 1947.304950702757875648075 I,

    10021.98853647687360689210 + 1947.433689663059592193412 I,

    10022.74359299516627559224 + 1947.562426455984204809882 I,

    10023.49864285270786185070 + 1947.691161081768026142988 I,

    10024.25368605017376660312 + 1947.819893540647322607883 I,

    10025.00872258823926404804 + 1947.948623832858314402362 I,

    10025.76375246757950168151 + 1948.077351958637175519890 I,

\medskip

                     1st differences of imaginary parts

\medskip
                            0.128799742991666628943

                            0.128797568978986076436

                            0.128795395203917926090

                            0.128793221666415580975

                            0.128791048366432457323

                            0.128788875303921984513

                            0.128786702478837605074

                            0.128784529891132774676

                            0.128782357540760962124

                            0.128780185427675649360

                            0.128778013551830331449

                            0.128775841913178516579

                            0.128773670511673726060

                            0.128771499347269494303

                            0.128769328419919368861

                            0.128767157729576910329

                            0.128764987276195692462

                            0.128762817059729302075

                            0.128760647080131339083

                            0.128758477337355416479

                            0.128756307831355160341

                            0.128754138562084209818

                            0.128751969529496217172

                            0.128749800733544847512

                            0.128747632174183779447

                            0.128745463851366704198

                            0.128743295765047326262

                            0.128741127915179363130

                            0.128738960301716545337

                            0.128736792924612616470

                            0.128734625783821333106

                            0.128732458879296464895

                            0.128730292210991794479

                            0.128728125778861117528

\medskip
             2nd differences of imaginary parts of zeros (times minus $1$)
\medskip

                           $0.2174012680552507 	\times 10^{-5}$

                           $0.2173775068150346 	\times 10^{-5}$

                           $0.2173537502345115 	\times 10^{-5}$

                           $0.2173299983123652 	\times 10^{-5}$

                           $0.2173062510472810 	\times 10^{-5}$

                           $0.2172825084379439 	\times 10^{-5}$

                           $0.2172587704830398 	\times 10^{-5}$

                           $0.2172350371812552 	\times 10^{-5}$

                           $0.2172113085312764 	\times 10^{-5}$

                           $0.2171875845317911 	\times 10^{-5}$

                           $0.2171638651814870 	\times 10^{-5}$

                           $0.2171401504790519 	\times 10^{-5}$

                           $0.2171164404231757 	\times 10^{-5}$

                           $0.2170927350125442 	\times 10^{-5}$

                           $0.2170690342458532 	\times 10^{-5}$

                           $0.2170453381217867 	\times 10^{-5}$

                           $0.2170216466390387 	\times 10^{-5}$

                           $0.2169979597962992 	\times 10^{-5}$

                           $0.2169742775922604 	\times 10^{-5}$

                           $0.2169506000256138 	\times 10^{-5}$

                           $0.2169269270950523 	\times 10^{-5}$

                           $0.2169032587992646 	\times 10^{-5}$

                           $0.2168795951369660 	\times 10^{-5}$

                           $0.2168559361068065 	\times 10^{-5}$

                           $0.2168322817075249 	\times 10^{-5}$

                           $0.2168086319377936 	\times 10^{-5}$

                           $0.2167849867963132 	\times 10^{-5}$

                           $0.2167613462817793 	\times 10^{-5}$

                           $0.2167377103928867 	\times 10^{-5}$

                           $0.2167140791283364 	\times 10^{-5}$

                           $0.2166904524868211 	\times 10^{-5}$

                           $0.2166668304670416 	\times 10^{-5}$

                           $0.2166432130676951 	\times 10^{-5}$

%\newpage

%\section {Appendix B}

\bibliography{/home/jhaglund/books/qtcat/qtref}

\end{document}